\documentclass[12pt,titlepage,a4paper]{article}

\usepackage{amssymb,latexsym,amsfonts,amsmath,amsthm}  

{\makeatletter \@addtoreset{equation}{section}}

\newcommand{\diag}{\mathop{\mathrm{diag}}\nolimits}

\newtheorem{theorem}{Theorem}[section]
\newtheorem{lemma}[theorem]{Lemma}

\newcommand{\wt}{\widetilde}
\newcommand{\bl}{\bigl(}
\newcommand{\br}{\bigl)}
\newcommand{\Bl}{\Bigl(}
\newcommand{\Br}{\Bigl)}
\newcommand{\bgl}{\biggl(}
\newcommand{\bgr}{\biggr)}

\newcommand{\noi}{\noindent}

\newcommand{\sn}{\smallskip\noindent}
\newcommand{\bsk}{\bigskip}

\begin{document}

\title{Minimal state space realizations in Jacobson normal form}
\author{
Naoharu Ito\\
Department of Mathematical Education\\
Nara University of Education\\
Takabatake-cho, Nara-shi\\
Nara 630-8528\\
Japan
         \and 
Wiland Schmale\\
Fachbereich 6 Mathematik\\
Carl-von-Ossietzky-Universit\"at\\
D-26111 Oldenburg\\
Germany\\ 
\phantom{iii}\\
%
Harald~K. Wimmer\\
Mathematisches Institut\\
Universit\"at W\"urzburg\\
D-97074 W\"urzburg\\
Germany}
\date{2001}


\maketitle

\begin{abstract}

We derive a procedure for a minimal state space
realization of a rational transfer matrix over an arbitrary
field. The procedure is based on the Smith-McMillan form
and leads to  a state transition matrix 
in Jacobson normal form. 

\vspace{2cm}

\vspace{.5cm}

\vspace{.5cm}
\noindent
{\bf Running title:} Minimal realizations

\vspace{5cm}
\noindent
{\bf Address for Correspondence:}\\ 
H. Wimmer\\
Mathemat. Institut \\
Universit\"at W\"urzburg\\
Am Hubland\\
D-97074 W\"urzburg\\
Germany

\vspace{.2cm}
\flushleft{
{\textsf e-mail:}
~~\texttt{\small wimmer@mathematik.uni-wuerzburg.de} }\\
{\textsf Fax:}~~
+49 931 888\,\,46\,11\\
\end{abstract}

\section{Introduction}

Realization theory provides tools and techniques 
for a wide range of
applications of mathematical systems theory. 
In particular, state space
realizations are used for systems identification 
(Kalman and Declaris 1970),
 linear sequential circuits (Gill 1966)
 and amplifier circuit
 synthesis (Newcomb 1967).
During the last decade 
realization techniques of algebraic
systems theory have been playing an increasing role 
in convolutional coding  (Rosenthal 2001).
 Various types of 
realizations serve as first order representations
for convolutional codes and are the
basis 
for the construction of new codes 
of Rosenthal and York (1999).

Overviews over the literature 
by De Schutter (2000) and Datta (1980)
 show  three distinct
approaches to construct   minimal realizations of a 
strictly proper rational transfer matrix $T(s)$. 
The starting point for the first approach developed  by
Ho and Kalman (1966), Silverman (1971),
Eising and Hautus (1981),
is the impulse response written as
\[
   \sum_{\nu = 1} ^\infty \, C_{\nu} s^{-\nu}  =  T(s)  .
\]
A block Hankel matrix containing the Markov 
parameters $C_{\nu}$ is then
transformed in such a way that it produces a triple 
$(F,G,H)$ of a  minimal
realization
\begin{equation}\label{eq1}
   H(sI-F)^{-1}G =  T(s).
\end{equation}
The algorithms of the 
second approach, e.g., of Mayne (1968), Rosenbrock (1970), 
Datta (1980),
take advantage
of the fact that, according to  Kailath (1980: Chapter 2),
it is fairly easy 
to write down a non-minimal controllable (or observable)
 realization by inspection.
A minimal realization is then
obtained by extracting the unobservable (or uncontrollable) parts.
 With the
exception of Datta (1980) 
the two methodes described above 
are not designed to give a 
matrix $F$ in \eqref{eq1} in a canonical form. In general, this can
only be achieved by a third class of approaches 
which employ factorizations
and transformations of the transfer matrix $T(s)$. Kalman's 
(1965)
pioneering
paper 
belongs to this group, and also  Pace and Barnett (1974),
 Montes (1976)  and Coppel (1981). 
 By transforming the 
partial fraction components of a complex (or
real) transfer matrix $T(s)$ into Smith-McMillan form and then
 using
Taylor expansions, Kalman produced a minimal realization
\eqref{eq1} with $F$ being in Jordan normal form (or in real Jordan
normal form). That procedure is restricted 
to algebraically (or real)
closed fields. 

In this paper we are dealing with transfer matrices over an 
arbitrary
field $K$. 
We will adapt Kalman's approach to obtain a realization
\eqref{eq1} where $F$ is in Jacobson normal form. 
The motivation for
 our
study comes from applications of systems over finite fields such as
 linear
sequential circuits (Gill 1966).
According to
Massey and Sain (1967), Forney (1970), Rosenthal et al. (1996),
and  Rosenthal (2001),
  convolutional codes can be interpreted as linear sequential
 circuits.
Therefore 
we have developed our realization with the prospect
 of new constructions of codes  in the spirit of Rosenthal and
York (1999).

\section{The Jacobson normal form}

Let us briefly recall how the Jacobson normal form extends the
 concept of
Jordan normal form. Throughout this paper $p\in K[s]$
will be a fixed monic irreducible polynomial,
\[
   p(s) = s^n  + a_{n-1} s^{n-1} +  \dots  + a_0. 
\]
Let
  \begin{equation}  \label{eq2}
     C=C(p)=\left(
    \arraycolsep3pt
     \begin{array}{cccc}
     0      & 1    &           & 0\\
     \vdots &      &  \ddots   &  \vdots \\
     0      & 0    &  \cdots   &  1\\
     -a_0   & -a_1 &  \cdots   & -a_{n-1}
     \end{array}
    \right)
 \end{equation}
be the companion matrix associated with $p$. In particular, if
$p=s-\lambda$, then $C(p)=(\lambda)_{1\times 1}$.
Define
 \begin{equation}\label{eq3}
    V=
    \left(
     \arraycolsep3pt
     \begin{array}{cccc}
     0      & 0    &  \cdots   & 0\\
     \vdots & \vdots     &     &  \vdots \\
     0      & 0    &  \cdots   &  0\\
     1      & 0    &  \cdots   & 0
     \end{array}
    \right)_{n\times n}
              = e_n e_1^T,
  \end{equation}
where $e_1 = (1,0\ldots, 0)^T$ and $e_n = (0,\ldots, 0,1)^T$are unit
vectors of $K^n$. We call
   \begin{equation} \label{eq4}
  J =  J(p^k) =
\left(
\begin{matrix}
    C  &    V    &        &             &      &    \\
       &   C   &    .     &           &        &      \\
       &      &    .     &   .        &       &      \\
       &      &         &    .      &   .       &      \\
       &      &         &           &   C    & V \\
       &      &         &           &        & C
       \end{matrix} 
      \right)_{nk\times nk}
   \end{equation}
a \emph{Jacobson block} corresponding to $p^k$. 
The Jordan block
\[ 
    J[(s- \lambda)^k] = 
                          {\left(
     \arraycolsep3pt
     \begin{array}{ccccc}
     \lambda & 1       &  0     & \cdots  & 0       \\
      0      & \lambda &  1     &         & \cdot   \\[-1ex]
      \vdots &         &        &  \ddots           \\[-1ex]
       \cdot &         &        &         & 1       \\
      0      & \cdot   & \cdot  & \cdots  & \lambda
     \end{array}
    \right)}_{k\times k}
   \]
is a special case of \eqref{eq4}.
The fact that the $(i,i+1)$-entries of $J$ are equal to 1
implies that $J$ is nonderogatory and that the
 Smith form of $sI-J$ is
$\diag (1,\dots 1,p^k)$. Hence, if $A\in K^{\ell \times \ell}$ 
has $p^k$ as its
only elementary divisor, then $A$ is similar (over $K$)
 to $J = J(p^k)$.
The following, more general result can be traced back to
 Krull's (1921) Ph.D. thesis.
More easily accessible references are the books of 
 Jacobson (1953),
Ayres (1962)   or Cohn (1974).
\begin{theorem} \label{theo.jac}
Let 
            $p_1,\dots , p_m$ 
be the distinct irreducible factors of the
characteristic polynomial of a
 matrix $A \in K^{\ell \times \ell}$ and 
let 
\begin{multline}
   p^{k_{11}}_1, \dots , p_1 ^{ k_{ 1  \tau _1}  } , \dots,
  p_m ^{k_{m1} }, \dots,     p_m ^{k_ {m \tau _m  } }, \\
 k_{11} \leq  \dots \leq 
k_{ 1  \tau _1} , \dots,  
   k_{m1} \leq \dots \leq  k_ {m \tau _m}  ,
\end{multline}
be  the corresponding elementary divisors. 
Then $A$ is similar to
  \begin{equation} \label{eq5}
   \diag \Bl J \bl p _1 ^{k_{11} } \br, \dots, J \bl p_m ^{ k_ {m \tau _m} }  
\br \Br \; .
  \end{equation}
\end{theorem}
The matrix \eqref{eq5} is called the \emph{Jacobson normal form} of $A$.

\section{Notation}

Let $K(s)$ be the field of rational functions over $K$. An element $f\in
K(s)$ is called {\em strictly proper} if $f= 0$ or 
          $f=g/h, \,\, g,h \in K[s], \,\, gh \neq 0$ and 
$\deg g < \deg h$.
Let $K_{sp}(s)$ be the
$K$-vector space of strictly proper rational 
functions over $K$. Then each
$f \in K(s)$ can be decomposed uniquely as
\[
  f = w+y
\]
such that $w\in
K_{sp}(s)$ and $y\in K[s]$. If we set $\pi_{-}f = w$, then $\pi_{-}$
is the
projection of $K(s)$ onto $K_{sp}(s)$. In a natural way these
definitions
extend elementwise
 to vectors and matrices of rational functions. For a 
nonzero polynomial
vector $h= (h_1,\dots h_r)^T\in K^r[s]$ 
we define
\[
  \deg h = \max \{ \deg h_i,\; h_i\ne 0, \;i = 1,\dots r\}\; .
\]
 We set $\deg h=-\infty$ if $h=0$.

Let $I_k$ denote the $k\times k$ identity matrix and 
define 
\[
  N_k =
        {\left(
    \arraycolsep3pt
     \begin{array}{ccccc}
         0   & 1       &  0     & \cdots   & 0      \\
      \cdot  & 0       &  1     &          &  \cdot \\[-1ex]
      \vdots &         &        &  \ddots           \\[-1ex]
       \cdot &         &        &          & 1      \\
        0    &  \cdot  & \cdot  & \cdots   & 0
     \end{array}
    \right)}_{k\times k} .
\] 

According to Horn and Johnson (1991: Chapter 4)
 the Kronecker product  
 of two matrices $A=(a_{ij})$ and $B$ is the block matrix 
\[
  A\otimes B = (a_{ij}\, B).
\]  
 Note that the Jacobson block \eqref{eq4} can be written
as 
   $J = I_k  \otimes C + N_k  \otimes V$.
If the products $AC$ and $BD$ exist then
  \begin{equation} \label{eq6} 
    (A \otimes B)(C \otimes D) = (AC \otimes BD)\; .
  \end{equation}

\section{A special case}

In this section we shall focus on a particular type of 
transfer matrices.
The general realization problem will then be reduced to that 
special case.
Let $W\in K^{q\times t}(s)$, $W\ne 0$, 
be of rank $1$,
  \begin{equation}\label{eq7}
   W = h \frac{1}{p^k} g^T
 \end{equation}
where $h\in K^q[s]$, $g\in K^t[s]$ are
polynomial vectors.
We want to construct a realization of $\pi_{-}W$. In addition to the
companion matrix $C$ associated with
$
  p(s) =
 s^n + a_{n-1}s^{n-1} + \dots +  a_0 
$
we shall need the matrix
\begin{equation}  \label{eqM}
  M = M(p) =
  \left(
   \arraycolsep3pt
  \begin{array}{ccccccc}
   a_1     & a_2   & \cdot & \cdot & \cdot & a_{n-1} & 1\\[1ex]
   a_2     & a_3   & \cdot & \cdot & \cdot &1        & 0 \\[0ex]
\cdot   & \cdot &       &       & \cdot & \cdot   & \cdot \\[-1ex]
 \cdot   & \cdot &       & \cdot &       & \cdot   & \cdot \\[-1ex]
   \cdot   & \cdot & \cdot &       &       & \cdot   & \cdot \\
   a_{n-1} & 1     & \cdot & \cdot & \cdot & 0       & 0  \\[1ex]
   1       & 0     & \cdot & \cdot & \cdot & 0       & 0
  \end{array}
  \right) \,.
\end{equation}
Note that $M$ satisfies 
  \begin{equation}\label{eq8}
   MC = C^T M\; .
  \end{equation}
Let $h$ have the $p$-adic expansion
  \begin{equation}  \label{eq9}
   h = h_0 + h_1 p +\dots + h_{k-1} p^{k-1}+ \dots
  \end{equation}
where
  \begin{equation} \label{eq10}
   h_i \in  K^q[s], \, \deg h_i < n = \deg p ,\; i \geq 0 \;. 
  \end{equation}
We define $H_i\in K^{q\times n}$, 
                                   $ i \geq 0 $,   by
   \begin{equation}\label{eq11}
    h_i(s) = H_i \begin{pmatrix} 1\\s\\ \vdots\\ s^{n-1}\end{pmatrix}\; .
   \end{equation}
Using the expansion
   \begin{equation}    \label{eq12}
    g = g_0 + g_1 p + \dots + g_{k - 1} p^{k-1} + \dots
   \end{equation}
with
   \begin{equation}\label{eq13}
   g_i\in K^t[s], \, \deg g_i<n, \; i \geq  0 \;,  
   \end{equation}
we define matrices $G_i\in K^{t\times n}$ by
   \begin{equation}\label{eq14}
     g_i(s) = G_iM
    \begin{pmatrix} 1\\s\\ \vdots\\ s^{n-1}\end{pmatrix}\; .
   \end{equation}

\begin{theorem}  \label{theo4.1}
Let $h\in K^q[s]$ and $g\in K^t[s]$ be given. Assume that
\begin{equation}  \label{eq14b}
   W = h\frac{1}{p^k}\, g^T
\end{equation}
is a coprime factorization. Let $H_i$ and $G_i$, $i=0,1,\dots, k-1$, be
defined by \eqref{eq9} -- \eqref{eq11} and \eqref{eq12} -- \eqref{eq14}.
Set
\[
  H = (H_0,\ldots, H_{k-1})
\]
and
\[
  G = \begin{pmatrix}
      G^T_{k-1}\\ \vdots\\[0.5ex] G^T_0
      \end{pmatrix} \; .
\]
Then
  \begin{equation} \label{eq15}
  \pi_{-} W = H \Bl sI-J\bl p^k\br\Br^{-1} G\; ,
  \end{equation}
and the realization in \eqref{eq15} is minimal.
\end{theorem}

Let us briefly describe how in the case of 
$p = s- \lambda$ the realization \eqref{eq15}
reduces to  the realization of Kalman 
(1965: 532-533).
Consider \eqref{eq7} with 
\[ 
  W = h \, \frac{1}{(s- \lambda)^k} \, g^T
\]
and 
\[ h = \sum _{i \geq 0} h_i(s- \lambda)^i, \,\,  h_i \in
  K^q, \quad \mathrm{and} \quad  
       g=  \sum _{i \geq 0} g_i(s- \lambda)^i, \,\, 
g_i \in   K^t.  
\]
   Because of 
    $\mathrm{deg}\,p = 1$ the matrix $M$ in
 \eqref{eqM} reduces to $M = I_1$. Furthermore,
in \eqref{eq11} and \eqref{eq14}
 we have 
   $h_i(s) = h_i$ and $g_i(s) = g_i$. Therefore
 \eqref{eq15} yields  
\[
\pi _{-}  h \frac{1}{(s- \lambda)^k}g^T     =
    (h_0, \dots , h_{k-1}) \Bigl( (s- \lambda) I_k -
     N_k \Bigr)^{-1} 
  \begin{pmatrix}
      g^T_{k-1}\\ \vdots\\[0.5ex] g^T_0
      \end{pmatrix} \; .
\]

The proof of  Theorem \ref{theo4.1}
 is based on two lemmas.

\begin{lemma}   \label{lemma4.2}
Let the polynomial vector $b\in K^n[s]$ be defined as
\[
 b(s) = (1,s,\ldots, s^{n-1})^T \; .
\]
Let $C = C(p)$ be the companion matrix for  the polynomial $p$
and let $M = M(p)$ be given by \eqref{eqM}.
Then
  \begin{equation}\label{eq16}
  (sI-C)^{-1} = \pi_{-}p^{-1} b b^T M
  \end{equation}
and
  \begin{equation}\label{eq17}
   \bl sI-J(p^k)\br^{-1} =
 \pi_{-}\bl p I_k-N_k\br^{-1}\otimes\, b b^T
    M\; .
  \end{equation}
\end{lemma}

\begin{proof} 
   Obviously  \, $(sI - C)b = p\, e_n$\, is equivalent to
  \begin{equation}\label{eq18}
  (sI-C)^{-1} e_n =  p^{-1} b.
  \end{equation}
From  \eqref{eq18} 
and \eqref{eq8} we obtain
  \begin{equation}\label{eq19}
   e^T_1 (sI-C)^{-1} = e^T_n M(sI-C)^{-1}  e_n^T(sI-C^T)^{-1} M = p^{-1}
                       b^TM\; .
  \end{equation}
It is easy to see that 
\[
s^j(sI-C)^{-1} = (s^{j-1}I + \cdots + C^{j-1}) +
       C^j(s^{-1}I + s^{-2}C + \cdots )
\]
implies 
  \begin{equation}\label{eq20}
   \pi_{-}s^j (sI-C)^{-1} = C^j (sI-C)^{-1}\; .
  \end{equation}
Therefore,
  \begin{multline} \label{eq21}
   \pi_{-}p^{-1} bb^TM   = \pi_{-} be^T_1 (sI-C)^{-1}= \\[2ex]
                         = \pi_{-} \sum^n_{\nu=1} s^{\nu-1}e_\nu e^T_1
                            (sI-C)^{-1} = \sum^n_{\nu=1} e_\nu e^T_1 C^{\nu-1}
                             (sI-C)^{-1} = \\[2ex]
           = \bgl \sum^n_{\nu =1} e_\nu e^T_\nu \bgr
                              (sI-C)^{-1} = (sI-C)^{-1}\; .\hspace*{3.4cm}
  \end{multline}
To verify \eqref{eq17} we note that
   \begin{equation} \label{eq22}
      \bl pI_k - N_k\br^{-1} =
        \left(
    \arraycolsep3pt
     \begin{array}{cccc}
     p^{-1} &  p^{-2} & \cdots &   p^{-k}\\
     0  & p^{-1} &\cdots  &   p^{-(k-1)} \\[1ex]
     \vdots  & \vdots & \ddots   & \vdots \\[1ex]
     0 & 0  & \cdots  & p^{-1}
     \end{array}
    \right) \;.
 \end{equation}
Put $Q=V(sI-C)^{-1}$ where $V$ is given by \eqref{eq3}. 
Then
\[
(sI-J)^{-1} =
        \left(
    \arraycolsep3pt
     \begin{array}{cccc}
    (sI-C)^{-1}   &  (sI-C)^{-1}Q   & \cdots &  
                                   (sI-C)^{-1}Q^{k-1} \\
     O  & (sI-C)^{-1}  &\cdots  &  (sI-C)^{-1}Q^{k-2}  \\[1ex]
     \vdots  & \vdots & \ddots   & \vdots \\[1ex]
     O & O  & \cdots  &  (sI-C)^{-1}
     \end{array}
    \right) \, .
 \]
Now \eqref{eq18},  \eqref{eq19} and
\[
   e^T_1 (sI-C)^{-1} e_n= p^{-1}
\]
imply 
\begin{multline*}
 (sI-C)^{-1}Q^{i-1}   =
 (sI-C)^{-1} \big[ e_n e^T_1 (sI-C)^{-1} \big] ^{i-1}
     = \\ 
\shoveleft    
    = (sI-C)^{-1} e_n \,\,  \big[ e^T_1 (sI-C)^{-1} e_n \cdots
     e^T_1 (sI-C)^{-1} e_n   \big]\,\,  e^T_1 (sI-C)^{-1} = 
 \\ 
 \shoveleft = 
      (sI-C)^{-1} e_n \,\,[ p^{-i+2} ]\,\,  e^T_1 (sI-C)^{-1} =
      p^{-1} b \,\,   p^{-i+2} \,\,  p^{-1} b^T M = \\
 = p^{-i}bb^TM = \pi_{-}\,p^{-i} bb^TM,\; i= 2,\ldots, k\; .
\end{multline*}
Hence, \eqref{eq16} and \eqref{eq22} 
 together with the definition of the Kronecker product
 yield
  \begin{align*}
   (sI-J)^{-1}  & = \pi_{-}
       \left(
    \arraycolsep3pt
     \begin{array}{cccc}
     p^{-1}bb^T M  &   \cdots &   p^{-k}bb^T M\\
       \vdots & \ddots  & \vdots \\[1ex] 
    0  & \cdots  &   p^{-1}bb^T M \\[1ex]
     \end{array}
    \right) 
= \\[3ex]
  & = \pi_{-} \bl (p I_k - N_k)^{-1} \otimes bb^T M\br \, .
  \end{align*}

\end{proof}

For the following well known result 
on the dimension of  minimal  realizations we refer to 
Coppel (1974).
\begin{lemma}\label{lemma4.3}
Let $P,S,A$ be polynomial matrices such that
\[
  R= PA^{-1}S
\]
is a coprime factorization. Then the dimension of a minimal
realization of $\pi_{-} R$ is equal to the degree of 
$\det A$.
\end{lemma}

\noi {\bf{Proof of Theorem 
                           4.1:} }               \\
From \eqref{eq9} and \eqref{eq12} follows
\[
\pi_{-}W = \pi_{-}    
  (h_0,\ldots, h_{k-1})
        \left(
    \arraycolsep3pt
     \begin{array}{cccc}
     p^{-1} &  p^{-2} & \cdots &   p^{-k}\\
     0  & p^{-1} &\cdots  &   p^{-(k-1)} \\[1ex]
     \vdots  & \vdots & \ddots   & \vdots \\[1ex]
     0 & 0  & \cdots  & p^{-1}
     \end{array}
    \right) 
  \,
  \begin{pmatrix}
    g^T_{k-1}\\ \vdots\\[1ex] g^T_0
  \end{pmatrix} \; .
\]
Hence the relations
\[
  (h_0,\ldots, h_{k-1}) = (H_0,\ldots, H_{k-1}) (I_k\otimes b)
\]
and
\[
  (g_{k-1},\ldots, g_0) = (G_{k-1},\ldots, G_0) (I_k \otimes Mb)\;,
\]
the identity \eqref{eq22}, and the product formula \eqref{eq6}
imply
  \begin{align*}
    \pi_{-}W  &  = 
 \pi_{-} \, H(I \otimes b) (pI - N )^{-1} (I \otimes b^T
                     M) G = \\[1ex]
              &  = H \pi_{-} \Big[  \bl p I -N \br ^{-1}
                     \otimes\, bb^T M   \Big] G \; .
 \end{align*}
Now \eqref{eq15} follows immediately form \eqref{eq17}.
According to Lemma \ref{lemma4.3} 
the realization \eqref{eq15} is minimal.
\hfill $\square$

\section{Reduction to realizations with a single Jacobson block}

Let $T\in K^{q\times t}(s)$, $T \ne 0$, be a 
strictly proper rational
matrix, let 
 $d \in K[s]$ be the monic least common denominator of all 
elements of $T$,
  and let $p_1,\ldots, p_m$ be monic irreducible polynomials 
such that
$p^{\ell_1}_1\ldots p^{\ell_m}_m$ is a prime factorization of $d$.
 To build a realization of $T$ based on Theorem \ref{theo4.1} we
carry out  two steps.
First we 
take a partial fraction decomposition of each entry of $T$.
Then we decompose  $T$ accordingly as
   \begin{equation}\label{eq23}
     T = \sum^m_{\mu = 1} T_{p_\mu}
   \end{equation}
such that 
 each component $T_{p_\mu}$ is strictly proper having only powers of
$p_\mu$ as denominators of its entries. If
   \begin{equation}\label{eq24}
H_\mu(sI-F_\mu)^{-1} G_\mu =  T_{p_\mu}(s),
    \; \mu = 1,\ldots, m\; ,
   \end{equation}
are minimal realizations and if we set
   \begin{equation*}
     F = \diag (F_1,\ldots, F_m), \; G = \begin{pmatrix} G_1\\\vdots\\G_m
         \end{pmatrix}, \; H = (H_1,\ldots, H_m)\; ,
   \end{equation*}
then
   \begin{equation}\label{eq25}
 H(sI-F)^{-1}G  = T(s) 
   \end{equation}
is a minimal realization. We call \eqref{eq25} the 
\emph{direct sum} of the
realizations \eqref{eq24}. 

At this point
 we may restrict ourselves to
a strictly proper rational matrix $T$ where the least common 
denominator of its entries is a power of an irreducible polynomial
$p$.
In the second reduction step we want to decompose 
such a matrix $T$ into
a sum of rank 1 matrices. Assume rank 
                                           $T=r$ and let
\[
   \Sigma = \begin{pmatrix} D & 0\\ 0 & 0 \end{pmatrix}
\]
be the Smith-McMillan form of $T$ with
\[
   D = \diag 
  \Bl \frac{a_1}{p^{k_1}}, \ldots, \frac{a_r}{p^{k_r}}\Br\; ,
        k_1 \ge \dots \ge k_r \ge  0 \;,
\]
 and 
             \[
a_i\in K[s], \,\, \mathrm{gcd}(p,a_i) = 1, \,\, 
                    i =
                            1,\ldots,r, \,\,
         a_1 | \dots | a_r \;. 
\]
\smallskip
Let 
       $U = (u_1,\ldots, u_q)\in K^{q\times q}[s]$ 
 and 
       $V=
               (v_1,\ldots , v_t) \in K^{t\times t}[s]$ 
be unimodular matrices such that
\begin{equation}  \label{eqSmth}
  T = U \Sigma V^T \; .
\end{equation}
It follows from Lemma \ref{lemma4.3} that a 
minimal realization of $T$ has dimension equal to
 $
     \sum _{i=1} ^r n k_i 
 $. 
Now let  $\wt{a}_i\in K[s]$ be such that
\[
  \pi_{-} \bigl( a_i p^{-k_i} \bigr) = \wt{a}_i  p^{-k_i} \; ,
\]
and define
\begin{equation} \label{eq26} 
           w_i = \bl u_i \wt{a}_i \br\, \frac{1}{p^{k_i}} \,
v^T_i,
 \; i= 1, \ldots, r \;.
\end{equation}
Then \eqref{eqSmth} and $T = \pi _{-} T$ imply
   \begin{equation}   \label{eq27}
     T = \sum ^r_{i=1} \pi_{-} w_i\; .
   \end{equation}
Clearly,  $ \pi_{-}  w_i = 0$ if $k_i = 0$. If $k_i > 0$
then \eqref{eq26} is a  coprime factorization
since $u_i$ and $v_i$ are
columns of unimodular matrices and 
$\mathrm{gcd}(p, \tilde{a}_i) = 1$.
In that case $n k_i$ is the dimension of a
minimal realization of $\pi_{-}  w_i$. 
Therefore a 
direct sum of minimal realizations of the matrices 
$\pi_{-}w_i$ yields a minimal
realization of $T$. We remark that \eqref{eq27} puts us in the
position to apply Theorem~\ref{theo4.1}.
\hfill $\square$

It has been pointed out by Gill (1966) that 
the resolvent
\begin{equation}  \label{eqRes}
   T(s) = 
              (sI - A)^{-1}
\end{equation}
of a matrix $A \in K^{\ell \times \ell}$ is a special case
of transfer matrix.
Thus our realization
algorithm applied to \eqref{eqRes} yields $ (sI - A)^{-1} =
H\,(sI - F)^{-1}\,G$ and $G = H^{-1}$. Hence $A =H\,F\, H^{-1}$,
and $H$ transforms $A$ into Jacobson normal form. A different
approach to derive the Jacobson normal form from the resolvent
is due to  Della Dora and Jung (1996).
%
\section{An example}

\parindent0pt
In the following example the underlying field is 
$K = \mathbb{Z}_5$.
We consider the transfer matrix 
\begin{multline}  \label{eqT}
T(s) = \\
 \left(
{\begin{array}{cc}
\dfrac{2\,s^{6} + 3\,s^{3} + 2\,s^{2} + s + 4}
{(s^{2} + s + 2)^{2}\, (s^{3} + 3\,s^{2} + s + 1)} & 
\dfrac { s^{6} + 4\,s^{3} + s^{2} + 2\,s  + 2}
{(s^{2} + s + 2)^{2}\,(s^{3} + 3\,s^{2} + s + 1)} \\[3ex]
\dfrac {2\,s^{6} + 3\,s^{3} + 2\,s^{2} + s + 1 }
{(s^{2} + s + 2)^{2}\,(s^{3} + 3\,s^{2} + s + 1)} 
& \dfrac {
2\,  ( 3\,s^{6} + 2\,s^{3} + 3\,s^{2}  + s + 3 ) }
  {(s^{2} + s + 2)^{2}\,(s^{3} + 3\,s^{2} + s  + 1)}
\end{array}}
 \right)
\end{multline}
      with entries in  $\mathbb{Z}_5(s)$.
We shall proceed along the lines of Section 5 and Section 4.

(1.) \underline{Partial fraction decomposition of $T$}\\
Let
$p_1 = s^2+s+2$
and
$p_2 = s^{3} + 3\,s^{2} + s + 1$.
Then
          \[ 
               T = T_{p_1} + T_{p_2}
   \]
and 
      \[
T_{p_1}(s) =
\left(
\begin{array}{cc}
\dfrac {3\,s^{3} + 4\,s^{2} + s  }
{(s^{2} + s + 2)^{2}} 
            & \dfrac { 3\,s^{3} + 2\, s^{2} + 3\,s + 4 }
{(s^{2} + s + 2)^{2}} \\[3ex]
\dfrac { s + 3}{(s^{2} + s + 2)^{2}} &
         \dfrac {2\,s^{3} + 4\,s ^{2} + 3\,s}
{(s^{2} + s + 2)^{2}}
\end{array}
 \right)
\]
and
\[
T_{p_2}(s) =  
\left(
\begin{array}{cc}
\dfrac {4\,s^{2} + 4\,s  + 1} 
{s^{3} + 3\,s^{2} + s + 1} &
   \dfrac {3\,s^{2}   + 3\,s + 2 }
        {s^{3} + 3\,s^{2} + s + 1} \\[3ex]
\dfrac {2\,s^{2} + s + 2}
{s^{3} + 3\,s^{2} + s + 1} & 
 \dfrac {4\,s^{2} + 2\,s + 4}
{s^{3} + 3\,s^{2} + s + 1}
\end{array}
 \right) \; .
\]
\bigskip

\noi (2.) \underline{Realization of $T_{p_1}$}
\medskip

The Smith-McMillan form of $T_{p_1}$ is
\[
\Sigma  =  
 \diag \bigl( \dfrac{a_1}{p^2}, \dfrac{a_2}{p} \bigr)
 =
\left(
{\begin{array}{cc}
\dfrac {1}{(s^{2} + s + 2)^{2}} & 0 \\[2ex]
0 & \dfrac {(s + 1)\,(s^{3} + 3\,s^{2} + 4)}{s^{2} + s + 2}
\end{array}}
 \right)\; .
 \]
We have $T_{p_1} = U \Sigma V^T$ \,, and the unimodular
matrices $U$ and $V$ are given by
\[
U=  (u_1, u_2) = 
\left(
{\begin{array}{cc}
s\,(3\,s^{2} + 4\,s + 1)
                        & 4\,s^{2} + 3 \\
s + 3 & 3
\end{array}}
 \right) , 
                  \]
and 
\[ 
      V=  (  v_1, \, v_2 ) =
\left(
{\begin{array}{cr}
1 & 0 \\
2\,s^{5} + 4\,s^{4} + s^{3} + 4\,s^{2} + 2 & 1
\end{array}}
 \right) \, .
\]
Note that 
\[
   \dfrac{a_2} {p} = \dfrac{(s + 1) (s^3 + 3s^2 + 4) }
  {s^2 + s + 2}
\]
is not strictly proper. We calculate $\tilde{a}_2$ 
and obtain
\[
  \pi _{-}  \dfrac{a_2} {p} =  \dfrac{\tilde{ a}_2} {p}
= 
   \dfrac{3} {s^2 + s + 2} \; .
\]
Set
\begin{multline*}
    w_1 = u_1 \,  \dfrac{\tilde{ a}_1} {p^2} \, v_1 ^T = \\
\begin{pmatrix} s\,(3\, s^2 + 4s + 1) \\
           3 +s 
       \end{pmatrix}\,
 \dfrac{1} {(s^2 + s + 2)^2} \, 
  \begin{pmatrix} 1 \\ 
 2\,s^{5} + 4\,s^{4} + s^{3} + 4\,s^{2} + 2
                                       \end{pmatrix} ^T
\end{multline*}
and
     \[ 
           w_2 
                  =
       u_2\,  \dfrac{\tilde{ a}_2} {p} \, v_2 ^T 
 =
     \begin{pmatrix} 4\,s^{2} + 3 \\
                                        3 
   \end{pmatrix} \, 
                    \dfrac{3} {s^2 + s + 2} \, 
\begin{pmatrix} 0 \\ 1 
                         \end{pmatrix} ^T \; .
\]
Then 
        $T_{p_1} = \pi _{-} w_1 +   \pi _{-} w_2$.   
        

\sn
\noi
(2.1) 
 \underline{Realization of $\pi _{-} w_{1}$.    }\\
\sn
We set 
    \[
           h 
                 =
                     u_1 \, \tilde{a}_1 =
    \begin{pmatrix}
                       s\,(3\, s^2 + 4s + 1) \\
           3 +s 
       \end{pmatrix} 
\]
  and
  \[ 
         g = v_1 
                     =  \begin{pmatrix}
        1 \\
               2\,s^{5} + 4\,s^{4} + s^{3} + 4\,s^{2} + 2
 \end{pmatrix}. \; 
\]
Then $  h = h_0 + h_1 p $\,
with
\[
 h_0 = \begin{pmatrix} 4s + 3
 \\ s + 3
\end{pmatrix}, \;  
        h_1 = 
\begin{pmatrix}  3s + 1 \\ 0
\end{pmatrix}. \; 
\]
Similarly, \, $   g = g_0 + g_1 \, p + g_2 \, p^2$ \,
with
  \[
     g_0 = \begin{pmatrix} 1 \\ 2 
\end{pmatrix}, \;  
                      g_1 =
       \begin{pmatrix}    0 \\ s 
\end{pmatrix}, \; 
                    g_2 =
        \begin{pmatrix}    0 \\ 2\,s 
\end{pmatrix}. \; 
\]

This leads to 
\[
H_ 0= \left(
{\begin{array}{rr}
3 & 4 \\
3 & 1
\end{array}}
 \right) , \,   H_1= \left(
{\begin{array}{rr}
1 & 3 \\
0 & 0
\end{array}}
 \right) ,
\]
    and
 \[
          (H_0 \, | \, H_1) =
 \left(
{\begin{array}{rr | rr}
3 & 4 & 1 & 3 \\
3 & 1 & 0 & 0 
\end{array}} 
 \right) =  H .
\]
The matrix $M$ in \eqref{eqM} is given by
\[  
         M = 
 \left(
{\begin{array}{rr} 1 & 1 \\
                                 1 & 0
     \end{array}}
 \right).
\]
Hence
\[
G_0= \left(
{\begin{array}{rr}
0 & 1 \\
0 & 2
\end{array}}
 \right) , \,G_1 = \left(
{\begin{array}{rc}
0 & 0 \\
1 & -1
\end{array}}
 \right)
                  \]
such that
 \[
       \left( {\begin{array}{c}
 
                                 G_1 ^T  \vspace{.5mm}
\\ \hline   G_0 ^T 
             \end{array}}
 \right) 
         =
              \left(
{\begin{array}{rc}
0 & 1 \\
0 & -1 \\ \hline
0 & 0 \\
1 & 2
\end{array}}
 \right) 
               =  G \; .
\]
Finally, for $p = s^2 + s + 2$,  we have 
\[
      J(p^2) = 
                        \left(
{\begin{array}{rr  rr}
0 & 1 &  \phantom{0} &  \phantom{0} \\
3 & 4 & 1 & \phantom{0} \\ 
\phantom{0} & \phantom{0} & 0 & 1 \\
\phantom{0} & \phantom{0} & 3 & 4
\end{array}}
 \right) 
                = F \, .
\]

\bsk
(2.2)  \underline{ Realization of $\pi_{-}w_{2}$ }\\
\smallskip
Set 
\[ 
        h = 
              u_2 \, \tilde{a}_2 = 
 \begin{pmatrix} 
                   4\, s^2 + 3 \\
          3 
 \end{pmatrix} \times 3   = 
\begin{pmatrix}         2 \, s^2 + 4 \\
                 4 
\end{pmatrix}
\]
       and
\[
     g = v_2 = \binom{0}{1} \, .
\]
Then \, $ h = h_0 + h_1 \, p$ \,\, with
\[
       h_0 = \begin{pmatrix} 3\, s \\ 4 \end{pmatrix}  , \;
   h_1 =
            \begin{pmatrix} 2 \\ 0\end{pmatrix} \, .
 \] 
Hence
\[
   H_0 = 
 \left(
{\begin{array}{rr}
0 & 3 \\
4 & 0
\end{array}}
 \right)
               = H \, .
\]
From 
\[ 
    g_0 =        \left(
{\begin{array}{r}
0  \\
1
\end{array}}
 \right) 
                \]
and \eqref{eq14}  we obtain        
      \[
 G_0 = 
 \left(
{\begin{array}{rr}
0 & 0 \\
0 & 1
\end{array}}
 \right)      
                  = G     \; .
\]
The corresponding state space matrix is
\[
    J(p^1) = C(p) = 
\begin{pmatrix}
0 & 1 \\
3 & 4
\end{pmatrix}
           = F \, .
\]

\bsk
(3.) \underline{  Realization of $T_{p_2}$ }

\sn
The Smith-McMillan form of $T_{p_2}$ is
\[
{\Sigma } =  \left(
{\begin{array}{cr}
\dfrac {1}{s^{3} + 3\,s^{2} + s + 1} & 0 \\
0 & 0
\end{array}}
 \right) \; .
\]
The unimodular matrices $U$, $V$ in the decomposition
$ U \Sigma V^T = T_{p_2}$ are
\[
U=   (u_1, u_2) =
                    \left(
{\begin{array}{cc}
4\,s^{2} + 4\,s + 1  & s \\
 2\,s^{2} + s + 2 & 3\,s  + 1
\end{array}}
 \right) , \,      
              V= (v_1, v_2) =
  \left(
{\begin{array}{rr}
1 & 0 \\
2 & 1
\end{array}}
 \right) \, .
\]
Set 
\[ 
 p = s^3 + 3 \, s^2 + s + 1.
\]
Then
\[
T_{p_2} 
        =
  u_1 \, \frac{1}{p} \,  
                           v_1 ^T = 
     \left(
{\begin{array}{c}
               4\,s^{2} + 4\,s + 1  \\
 2\,s^{2} + s + 2
\end{array}}
 \right) \, \dfrac{1}{s^3 + 3s^2 + s + 1} 
                                           \begin{pmatrix} 1 \\ 0
                            \end{pmatrix} ^T \, .
\]
It is easy to see that one can obtain the minimal realization
of $T_{p_2}$ directly from Theorem \ref{theo4.1}. Note that 
 $T_{p_2}$ 
 is of the form \eqref{eq14b} 
with $ h = u_1$,  $   g= v_1$,
                   and $k=1$. 
Moreover, $\mathrm{deg}\, h < \mathrm{deg}\, p$ and
$\mathrm{deg}\, g < \mathrm{deg}\, p$ imply $H = H_0$ and
$G = G_0$.
Thus $h = h_0$ yields
\[
H_0= \left(
{\begin{array}{rrr}
1 & 4 & 4 \\
2 & 1 & 2
\end{array}}
 \right)
               = H \, .
\]
From 
\begin{equation} \label{eqMp}
               M = 
 \left( 
{\begin{array}{ccc}
                         1 & 3 & 1 \\
           3 & 1 & 0 \\
     1 & 0 & 0 
                        \end{array}} \right)
\end{equation}
and
  $g = g_0$ follows 
\[
G_0= \left(
{\begin{array}{rr}
0 & 0 \\
0 & 0 \\
1 & 2
\end{array}}
 \right)
 =G   \; .
\]
Finally, we have 
\[
        J(p^1) = C(p) = 
 \left(       {\begin{array}{rrr}
0 & 1 & 0 \\
0 & 0 & 1 \\
4 & 4 & 2
\end{array}}
 \right) 
              = F \, .
\]

\bsk
(4.) \underline{Realization of $T$}

\sn
Taking the direct sum of the realizations of 
$\pi _{-} w_1$, $\pi _{-} w_2$ and $T_{p_2}$ we obtain
\[         
F  =  \left(
{\begin{array}{rrrr rrrrr}
0 & 1 & 
 \phantom{0} & \phantom{0} \vline &  \phantom{0}
 & \phantom{0} & \phantom{0} & \phantom{0} & \phantom{0} \\
3 & 4 & 1 & \phantom{0} \vline & 
 \phantom{0} & \phantom{0} & \phantom{0} & 
\phantom{0} & \phantom{0} \\
0 & 0 & 0 & 1  \vline &  \phantom{0} & \phantom{0} & \phantom{0} & 
\phantom{0} & \phantom{0} \\
0 & 0 & 3 & 4 \vline& \phantom{0} & \phantom{0} & \phantom{0} & 
\phantom{0} & \phantom{0} \\
\cline{1-6} 
 \phantom{0} & \phantom{0} & \phantom{0} & \phantom{0} \vline &
  0 & 1  \vline&  \phantom{0} & \phantom{0} & \phantom{0} \\
 \phantom{0} & \phantom{0} & \phantom{0} & \phantom{0}  \vline &
  3 & 4  \vline &  \phantom{0} & \phantom{0} & \phantom{0} \\
\cline{5-9}
 \phantom{0} & \phantom{0} & \phantom{0} & 
\phantom{0} & \phantom{0} & \phantom{0} \vline &
0 & 1 & 0 \\
\phantom{0} & \phantom{0} & \phantom{0} & 
\phantom{0} & \phantom{0} & \phantom{0} \vline &
0 & 0 & 1 \\
\phantom{0} & \phantom{0} & \phantom{0} & 
\phantom{0} & \phantom{0} & \phantom{0} \vline &
4 & 4 & 2
\end{array}}
 \right) \, ,
  \]
\[
H= \left(
{\begin{array}{rrrr | rr |rrr}
3 & 4 & 1 & 3 & 0 & 3 & 1 & 4 & 4 \\
3 & 1 & 0 & 0 & 4 & 0 & 2 & 1 & 2
\end{array}}
 \right) \quad \mathrm{and} \quad
G= \left(
{   \begin{array}{rc}
0 & 1 \\
0 & -1 \\
0 & 0 \\
1 & 2 \\ \hline
0 & 0 \\
0 & 1 \\ \hline
0 & 0 \\
0 & 0 \\
1 & 2
            \end{array}   }
 \right) \; . 
                                              \]
Then the transfer matrix $T$ in \eqref{eqT} 
has a minimal realization $H(sI - F)^{-1} G = T(s)$
where  the matrices $F, G, H$ are the ones 
displayed above, and $F$ is in Jacobson normal form. 

\bigskip \medskip

\noindent {\bf{References}}

\bigskip


 \noindent
{\sc Ayres, F.}, JR., 1962,
 {\em Schaum's Outline of 
Theory and Problems of Matrices}
(New York: McGraw Hill).

\bigskip \noindent
{\sc Cohn, P.M.}, 1974, {\em Algebra, Vol.\,1}
 (London: Wiley).

\bigskip \noindent
{\sc Coppel, W.A.}, 1981,
Linear systems: Some algebraic aspects.
{\em Linear Algebra Appl.}, {\bf{40}}, 257--273.

\bigskip \noindent
{\sc Coppel, W.A.}, 1974, Matrices of rational functions. 
{\em Bull. Austral. Math. Soc.}, {\bf{11}}, 89--113.

\bigskip \noindent

{\sc Datta, K.B.}, 1980,
 Minimal realizations in companion forms.
{\em J.~Franklin Inst.},  {\bf{309}}, 103--123.

\bigskip \noindent
{\sc Della Dora, J.,} and {\sc Jung, F.}, 1996,
Resolvent and rational canonical forms of matrices.
{\em SIGSAM Bull.}, {\bf{30(117)}}, 4--10.

\bigskip \noindent
{\sc De Schutter, B.}, 2000, 
Minimal state-space realization in linear 
system theory: An overview.
{\em J.~Comput. Appl. Math.}, {\bf{121}}, 
331--354.

\bigskip \noindent
{\sc Eising, R.}, and {\sc Hautus,  M.L.J.}, 1981,
 Realization algorithms for systems over 
a principal ideal domain.
{\em Math. Syst. Theory}, {\bf{14}}, 353--366. 

\bigskip \noindent
{\sc Forney, G.D.,} JR., 1970,
Convolutional codes, I, Algebraic structure. 
 {\em{
IEEE Trans. Inform. Theory}}, {\bf{IT-16}},  720--738.

\bigskip \noindent
{\sc Gill, A}, 1966, 
{\em Linear Sequential Circuits} (New York: McGraw-Hill).

\bigskip \noindent
{\sc Ho, B.L.}, and {\sc Kalman, R.E.}, 1966,
Effective construction of linear state
variable models from input/output functions.
{\em Regelungstechnik}, {\bf{14}},
545--548.

\bigskip \noindent
{\sc Horn, R.A.}, and
{\sc Johnson}, CH.R., 1991, {\em Topics in Matrix Analysis}
 (Cambridge: Cambridge University Press).

 \bigskip \noindent
{\sc Jacobson, N.}, 1953,
 {\em Lectures in Abstract Algebra, Vol. II - 
Linear Algebra} (Princeton:  Van Nostrand).


 \bigskip \noindent
{\sc Kailath, Th.}, 1980, {\em Linear Systems} 
 (Englewood Cliffs: Prentice Hall).

 \bigskip \noindent
{\sc Kalman, R.E.}, 1965,
Irreducible realizations  and the degree of a
rational matrix. 
{\em J.\ Soc.\ Ind.\ Appl.\ Math.},
{\bf{13}}, 520--544.

 \bigskip \noindent
{\sc Kalman, R.E}, and {\sc Declaris, N.}, 1970,
 editors,  
{\em Aspects of Network and System
Theory} (New York: Holt, Reinhart \& Winston).

\bigskip \noindent 
{\sc Krull, W.}, 1921,
 {\em \"Uber Begleitmatrizen und 
Elementarteilertheorie},
Dissertation, Freiburg, 
 Gesammelte Abhandlungen, Bd. 1, 
edited by P. Ribenboim, 1999 (Berlin: de Gruyter), pp.\,55--95.

\bigskip \noindent
{\sc Massey,  J.L.}, and {\sc Sain, M.K.}, 1967,
Codes automata and continuous
systems:
Explicit interconnections. 
{\em IEEE Trans.\ Automat.\ Control},{\bf{\ AC--12}},
644--650.

\bigskip \noindent
{\sc Mayne, D.Q.}, 1968,
 Computational procedure for the 
minimum realization
of transfer function matrices.
{\em Proc.\ IEE}, {\bf{115}}, 1363--1368.

\bigskip \noindent
{\sc Montes, C.G.}, 1976,
 Minimal realization of a transfer 
function matrix. 
 {\em IEEE Trans.\ Automat.\ Control},
 {\bf{AC--21}}, 399--401.

\bigskip \noindent
{\sc Newcomb, R.W}, 1967,
 {\em Active Integrated Circuit Synthesis} 
 (Englewood Cliffs: Prentice Hall). 

\bigskip \noindent
{\sc Pace, I.S.} and {\sc Barnett, St.}, 1974,
Efficient algorithms for linear
system calculations, II: Minimal realizations. 
{\em Int.\ J.~Systems Sci.}, {\bf{5}},
413--424.

\bigskip \noindent
{\sc Rosenbrock, H.H}, 1970,
{\em State-Space and Multivariable Theory}
(New York: Wiley).

\bigskip \noindent
{\sc Rosenthal, J.}, 2001,
 Connections between linear systems and 
convolutional codes.
{\em Codes, Systems and 
Graphical Models}. IMA Vol. 123, edited by 
B.~Marcus and J.~Rosenthal (New York: Springer-Verlag),
pp.~39--66.

\bigskip \noindent
{\sc Rosenthal, J., Schumacher, J.M.,} and 
{\sc York, E.V.}, 1996,
On behaviors and convolutional codes.
{\em IEEE Trans. Inform. Theory}, {\bf{IT-42}}, 
  1881--1991.

\bigskip \noindent
{\sc Rosenthal, J.}, and {\sc York, E.V.}, 1999,
 BCH convolutional codes.
 {\em IEEE Trans. Inform. Theory}, {\bf{IT-45}}, 
1833--1844.

\bigskip \noindent
{\sc Silverman, L.M.}, 1971,
Realization of linear dynamical systems.
{\em IEEE Trans.\ Automat. Control}, {\bf{AC-16}},
 554--567.

\end{document}